\documentclass[12pt]{article}
\usepackage{amsmath, amsthm}
\usepackage{enumerate}
\usepackage{amssymb}
\usepackage{geometry}
\newtheorem{thm}{Theorem}[section]
\newtheorem{exam}{Example}[section]
\newtheorem{cor}[thm]{Corollary}

\newtheorem{prop}[thm]{Proposition}

\theoremstyle{definition}
\newtheorem{defn}{Definition}[section]
\theoremstyle{remark}
\newtheorem{rem}{Remark}[section]

\DeclareMathOperator{\tc}{\xrightarrow[]{\tau}}
\DeclareMathOperator{\tcc}{\xrightarrow[]{\acute{\tau}}}

\newcommand{\eval}[2][\right]{\relax
  \ifx#1\right\relax \left.\fi#2#1\rvert}

\begin{document}
\title{\bf Convergence via filter in locally solid Riesz spaces} 
\maketitle

\author{\centering Abdullah AYDIN \\ \bigskip  \small  
	Department of Mathematics, Mu\c{s} Alparslan University, Mu\c{s}, Turkey. \\}

\bigskip

\abstract{Let $(E,\tau)$ be a locally solid vector lattice. A filter $\mathcal{F}$ on the set $E$ is said to be converge to a vector $e\in E$ if, each zero neighborhood set $U$ containing $e$, $U$ belongs to $\mathcal{F}$. We study on the concept of this convergence and give some basic properties of it. 

\bigskip
\let\thefootnote\relax\footnotetext
{Keywords: filter, locally solid Riesz space, vector lattice 
	
\text{2010 AMS Mathematics Subject Classification:} 46A40, 46A55

e-mail: a.aydin@alparslan.edu.tr}

\section{Introduction and Main Results}
Let give some basic notation and terminology that will be used in this paper. Let $E$ be real vector space. Then $E$ is called ordered vector space if it has an order relation $\leq$ that means it is reflexive, antisymmetric and transitive, and also it satisfies the axioms; if $y\leq x$ then $y+z\leq x+z$ for all $z\in E$, and if $y\leq x$ then $\lambda y\leq \lambda x$ for all $0\leq \lambda$. 
An ordered vector $E$ is said to be {\em vector lattice} (or, {\em Riesz space}) if, for each pair of vectors $x,y\in E$, the supremum $x\vee y=\sup\{x,y\}$ and the infimum $x\wedge y=\inf\{x,y\}$
both exist in E. Also, $x^+:=x\vee 0$, $x^-:=(-x)\vee0$, and $\lvert x\rvert:=x\vee(-x)$ are called the {\em positive} part, the {\em negative} part, and the {\em absolute value} of $x$, respectively. In a vector lattice $E$, a subset $A$ is called as {\em solid} if, for $y\in A$ and $x\in E$ with $\lvert x\rvert\leq\lvert y\rvert$, we have $x\in A$. Also, two vector $x$, $y$ in a vector lattice is said to be {\em disjoint} whenever $\lvert x\rvert\wedge\lvert y\rvert=0$; for more details see \cite{ABur,ABPO}.

Let $E$ be vector lattice $E$ and $\tau$ be a linear topology on it. Then $(E,\tau)$ is said a {\em locally solid vector lattice} (or, {\em locally solid Riesz space}) if $\tau$ has a base which takes place with solid sets, for more details on these notions see \cite{ABur,ABPO,AB}. It is known that every linear topology $\tau$ on a vector space $E$ has a base $\mathcal{N}$ for the zero neighborhoods satisfying the followings; for each $V\in \mathcal{N}$, we have $\lambda V\subseteq V$ for all scalar $\lvert \lambda\rvert\leq 1$; for any $V_1,V_2\in \mathcal{N}$ there is another $V\in \mathcal{N}$ such that $V\subseteq V_1\cap V_2$; for each $V\in \mathcal{N}$ there exists another $U\in \mathcal{N}$ with $U+U\subseteq V$; for any scalar $\lambda$ and for each $V\in \mathcal{N}$, the set $\lambda V$ is also in $\mathcal{N}$; see for example \cite{AA,Tr}. Hence, every locally solid vector lattice satisfies these properties. Also, it follows from \cite[Thm.2.28]{ABur} that a linear topology $\tau$ on a vector lattice $E$ is a locally solid iff it is generated by a family of Riesz pseudonorms $\{\rho_j\}_{j\in J}$, where {\em Riesz pseudonorm} is a realvalued map $\rho$ on a vector space $E$ if it satisfies the following conditions; $\rho(x)\geq0$ for all $x\in X$; if $x=0$ then $\rho(x)=0$; $\rho(x+y)\leq\rho(x)+\rho(y)$ for all $x,\ y\in X$; if $\lim\limits_{n\to\infty}\lambda_n=0$ in $\mathbb{R}$ then $\rho(\lambda_nx)\to 0$ in $\mathbb{R}$ for all $x\in X$; if $\lvert x\rvert\leq \lvert y\rvert$ then $\rho(x)\leq \rho(y)$. Moreover, if a family of Riesz pseudonorms generates a locally solid topology $\tau$ on a vector lattice $E$ then $x_\alpha \tc x$ iff $\rho_j(x_\alpha-x)\to 0$ in $\mathbb{R}$ for each $j\in J$. In this article, unless otherwise, the pair $(E,\tau)$ refers to as a locally solid vector lattice, and the topologies in locally solid vector lattices are generated by families of Riesz pseudonorms $\{\rho_j\}_{j\in J}$. In this paper, unless otherwise, when we mention a zero neighborhood, it means that it always belongs to a base that holds the above properties. Recently, there are some studies on locally solid Riesz spaces; see for example \cite{AA,L,Z}. 

Next, we give the concept of filters. Let $X$ be a set. A subset $\mathcal{F}$ of the power set of $X$ is said to be {\em filter} on $X$ if the followings hold; $\emptyset\notin \mathcal{F}$; if $A\in \mathcal{F}$ and $A\subseteq B$ then $B\in \mathcal{F}$; $\mathcal{F}$ is closed under finite intersections; see \cite{AB}. The second condition says that the set $X$ belongs to the filter on it. The filter can be defined tanks to its base. A nonempty subset $\mathcal{B}\subseteq\mathcal{F}$ is called a {\em filter base} for a filter $\mathcal{F}$, if $\mathcal{F}=\{F\subseteq X:\exists B\in\mathcal{B}, B\subseteq F\}$. A base $\mathcal{B}$ satisfies the following properties; $\mathcal{B}$ is nonempty; each $B\in\mathcal{B}$ is nonempty; for each $B_1, B_2\in\mathcal{B}$, there is another $B\in\mathcal{B}$ such that $B\subseteq B_1\cap B_2$. 

Filter can be defined with nets. A given partially ordered set $I$ is called {\em directed} if, for each $a_1,a_2\in I$, there is another $a\in I$ such that $a\geq a_1$ and $a\geq a_2$. A function from a directed set $I$ into a set $E$ is called a {\em net} in $E$. Now, we give a relation between net and filter convergence. Let $(x_\alpha)_{(\alpha\in I)}$ be a net in the set $E$. The filter $\mathcal{F}$ which is associated of $(x_\alpha)_{(\alpha\in I)}$ is defined as follows; let $\hat{x}_\alpha=\{x_\alpha:\alpha_0\in I, \alpha\geq\alpha_0\}$, and so the collection $\mathcal{B}=\{\hat{x}_\alpha:\alpha\in I\}$ is a filter base and the filter that is generated by $\mathcal{B}$ is the associated filter of $(x_\alpha)_{(\alpha\in A)}$. So, we can give the following natural example.
\begin{exam}
Let $E$ be a set and $(x_\alpha)_{(\alpha\in I)}$ be a net in $E$. The elementary filter associated to $(x_\alpha)_{(\alpha\in I)}$ is
$$
\mathcal{F}_{(x_\alpha)}=\{F\subseteq E:\exists \alpha_0, \ x_\alpha\in F\ \text{for all} \ \alpha\geq \alpha_0\}.
$$
\end{exam}

The filter convergence is defined on topological spaces with respect to the neighborhood of limit points; see \cite[Def.3.7]{J}. In here, we define this concept on locally solid vector lattices.
\begin{defn}
Let $(E,\tau)$ be a locally solid vector lattice and $\mathcal{F}$ be a filter on the set $E$. A vector $e\in E$ is said to be a {\em limit of $\mathcal{F}$} (or, $\mathcal{F}$ is said to  {\em converge to $e$}) if each zero neighborhood set containing $e$ belongs to $\mathcal{F}$, abbreviated as $\mathcal{F}\tc e$. Also,  a vector $x\in E$ is said to be a {\em cluster} point of $\mathcal{F}$ if each zero neighborhood set containing $x$ intersects every member of $\mathcal{F}$.
\end{defn}

\begin{exam}
Suppose $(E,\tau)$ is a locally solid vector lattice. For a vector $e\in E$, the filter generated by $e$ is 
$$
\mathcal{F}_e=\{U\subseteq E:U \ \text{is zero neighborhood and contains}\ e\}.
$$
So, we can see $\mathcal{F}\tc e$
\end{exam}
Consider a topological convergence net $(x_\alpha)_{(\alpha\in I)}\tc e\in E$. Then the elementary filter associated to $(x_\alpha)_{(\alpha\in I)}$ is also convergent to $e$ since every zero neighborhood containing $e$ belongs to that filter. 

\begin{rem}
Let $(E,\tau)$ be a locally solid vector lattice. Then

$(i)$ If a filter $\mathcal{F}\tc e$ then $e$ is a cluster point of $\mathcal{F}$. Indeed, let $U$ be zero neighborhood set and it contains $e$, and $F\in\mathcal{F}$. Since $U$ and $F$ in $\mathcal{F}$, we have $U\cap F\in \mathcal{F}$. Thus, we get $U\cap F\neq\emptyset$ since $\emptyset$ is not in the filter $\mathcal{F}$.

$(ii)$ Let $\mathcal{F}_1$, $\mathcal{F}_2$ be filters on $E$ with $\mathcal{F}_1\subseteq\mathcal{F}_2$. So, if $\mathcal{F}_1\tc e$ then $\mathcal{F}_2\tc e$ for $e\in E$. Indeed, every zero neighborhood containing $e$ is in $\mathcal{F}_1$, and so is in $\mathcal{F}_2$. Hence, $\mathcal{F}_2\tc e$.
\end{rem}

In the filter convergence, the limit point may not be unique. That means a filter $\mathcal{F}$ on a locally solid vector lattice $(E,\tau)$ can be convergence both $e_1, \ e_2\in E$.
\begin{prop}
Let $(E,\tau)$ be a locally solid vector lattice, and $\mathcal{F}$ be a filter on $E$. Then the followings hold;

\item[(i)] If $\mathcal{F}\tc e$ for some $e\in E$ then $\mathcal{F}\tc e+x$ for all $x \in E$ whenever $e$ and $x$ are positive, or disjoint;
		
\item[(ii)] If $\mathcal{F}\tc e$ for some $e\in E$ then $\mathcal{F}\tc \lvert e\rvert$

\item[(iii)]  If $\mathcal{F}\tc e$ for some $e\in E$ then $\mathcal{F}\tc e^+$ and $\mathcal{F}\tc e^-$
\end{prop}

\begin{proof}
$(i)$ Let $U$ be a zero neighborhood that contains $e+x$. So, we show $U\in \mathcal{F}$. Under the condition of positivity of $e$ and $x$, we have $e, x\in U$ since $U$ is a solid set and $e,x\leq e+x$. Therefore, by using the convergence $\mathcal{F}$ to $e$, we get $U\in \mathcal{F}$. On the other hand, if $e$ and $x$ are disjoint then we have $\lvert e+x\rvert=\lvert e\rvert+\lvert x\rvert$; see \cite[Thm.1.7(2-7)]{ABPO}. So, by solidness of $U$, we get $e, x\in U$. Therefore, $U\in \mathcal{F}$ since $\mathcal{F}\tc e$.

$(ii)$ Suppose $\mathcal{F}\tc e$ and $U$ is a zero neighborhood containing $\lvert e\rvert$. By the formula $\lvert e\rvert\leq \big\lvert \lvert e\rvert\big\rvert=\lvert e\rvert$ and by the solidness of $U$, we have $e\in U$. Therefore, $U\in \mathcal{F}$ since $\mathcal{F}\tc e$.

$(iii)$ Assume $\mathcal{F}\tc e$ and $U$ is a zero neighborhood which contains $e^+$. By the formula $\lvert e\rvert=e^++e^-$; see \cite[Thm.1.5(2)]{ABPO}, and by using the solidness of $U$, we get $e^+\in U$. Therefore, we have $U\in \mathcal{F}$ since $\mathcal{F}\tc e$. The other part of the proof is analog.
\end{proof}

In the next three results, we give some basic properties of filter convergence on locally solid vector lattices.
\begin{thm}\label{filter conv with addition}
Let $(E,\tau)$ be a locally solid vector lattice. If $\mathcal{F}_1$ and $\mathcal{F}_2$ are filters on $E$ such that $\mathcal{F}_1\tc e$ and $\mathcal{F}_2\tc x$ for some $e, x\in E$ then the set $\mathcal{F}=\{F_1\cup F_2:F_1\in\mathcal{F}_1\ \text{and}\ F_2\in\mathcal{F}_2\}$ is also a filter on the set $E$, and $\mathcal{F}\tc e+x$ whenever $e,\ x$ are positive, or disjoint. 
\end{thm}

\begin{proof}
Firstly, we show that $\mathcal{F}$ is a filter. $(i)$ $\emptyset \notin \mathcal{F}$ since $\emptyset \notin \mathcal{F}_1$ and $\emptyset \notin \mathcal{F}_2$; if $F_1\in\mathcal{F}_1$ and $F_2\in\mathcal{F}_2$, and $F_1\cup F_2\subseteq A\subseteq E$ then $A\in \mathcal{F}_1$ and $A\in \mathcal{F}_2$, so $A\in \mathcal{F}$; if $F_1,G_1\in\mathcal{F}_1$ and $F_2,G_2\in\mathcal{F}_2$ then $F_1\cup F_2$, $G_1\cup G_2$ and $(F_1\cup F_2)\cap(G_1\cup G_2)$ in both $\mathcal{F}_1$ and $\mathcal{F}_2$, and so $(F_1\cup F_2)\cap(G_1\cup G_2)\in \mathcal{F}$.

Next, for the first case, we show $\mathcal{F}\tc e+x$. Let $U$ be a zero neighborhood and it contains $e+x$. By the properties of zero neighborhoods in locally solid vector lattice, we have $e,\ x\in U$ since they are positive and so that $e,x\leq e+x$. Therefore, since $\mathcal{F}_1$ and $\mathcal{F}_1$ are filters on the set $E$, and also $\mathcal{F}_1\tc e$ and $\mathcal{F}_2\tc x$, we get $U\in \mathcal{F}_1$ and $U\in \mathcal{F}_2$. So, we get $U\in \mathcal{F}$.

Now, suppose $e$ and $x$ are disjoint in $E$. Then we have $\rvert e+x\rvert=\lvert e\rvert+\lvert x\rvert$; see \cite[Thm.1.7(2-7)]{ABPO}. So, for given any zero neighborhood $U$ containing $e+x$, we have $e, \ x \in U$. Then the proof follows like first case.
\end{proof}

\begin{thm}
Suppose $(E,\tau)$ is a locally solid vector lattice, and $\mathcal{F}_1$ and $\mathcal{F}_1$ are filters on the set $E$ such that $\mathcal{F}_1\tc e$ and $\mathcal{F}_2\tc x$ for some $e,x\in E$. Then the class $\mathcal{F}=\{F_1\cap F_2:F_1\in\mathcal{F}_1, F_2\in\mathcal{F}_2, \ \text{and} \ F_1\cap F_2\neq \emptyset \}$ is also a filter on $E$, and also $\mathcal{F}\tc e+x$  whenever $e,\ x$ are positive, or disjoint.
\end{thm}

\begin{proof}
We show $\mathcal{F}$ is a filter. $(i)$ $\emptyset \notin \mathcal{F}$; if $F_1\in\mathcal{F}_1$ and $F_2\in\mathcal{F}_2$, and $F_1\cap F_2\subseteq A\subseteq E$ then $A\in \mathcal{F}_1$ and $A\in \mathcal{F}_2$, so we get $A\in \mathcal{F}$; if $F_1,G_1\in\mathcal{F}_1$ and $F_2,G_2\in\mathcal{F}_2$ such that intersection of for each pair of them is non empty then $F_1\cap G_1\in\mathcal{F}_1$, $F_2\cap G_2\in\mathcal{F}_2$, so $(F_1\cap F_2)\cap(G_1\cap G_2)=(F_1\cap G_1)\cap(F_2\cap G_2)$ is non empty and it is also in $\mathcal{F}$. Then $\mathcal{F}$ is a filter on $E$, and by using the similar way in the proof of Theorem \ref{filter conv with addition} we get the desired result.
\end{proof}

Let $(x_\alpha)_{\alpha\in I}$ be a net in any topological space. A point $x$ is called {\em cluster point} of $(x_\alpha)_{\alpha\in I}$ if, for each neighborhood $U$ of $x$, $(x_\alpha)_{\alpha\in I}$ is frequently in $U$, or equivalently, for each neighborhood $U$ of $x$, we have $(U\setminus \{x\})\cap \{x_\alpha\}\neq  \emptyset$). On the other hand, let $\mathcal{F}$ be a filter on a locally solid vector lattice $(E,\tau)$. A vector $e\in E$ is said to be {\em cluster point of $\mathcal{F}$} (with respect to $\mathcal{F}$) if every zero neighborhood which contains $e$ intersects with each member of $\mathcal{F}$; see \cite[Def.3.7]{J}. So, we give the following result which has a relation between filter convergence and cluster point.
\begin{thm}
Let $(E,\tau)$ be a locally solid vector lattice and $\mathcal{F}$ be a filter on the set $E$. Then the vector $e\in E$ is a cluster point of $\mathcal{F}$ iff there exists another filter $\mathcal{F}_1$ containing $\mathcal{F}$ such that  $\mathcal{F}_1\tc e$.
\end{thm}

\begin{proof}
Suppose $e\in E$ is a cluster point of $\mathcal{F}$. Then the set $\mathcal{B}=\{U\cap F:F\in\mathcal{F}, U\ \text{is a zero neighborhood and contains}\ e\}$ is a filter base. Indeed, we show the properties of filter base. $(i)$ $\mathcal{B}$ is non empty since, for each zero neighborhood $U$ containing $e$ intersects with every member of $\mathcal{F}$; $(ii)$ for $U\cap F\in \mathcal{B}$, we have $U\cap F\neq \emptyset$ since $U\cap F\in \mathcal{F}$ and $\emptyset\notin \mathcal{F}$; $(iii)$ for $U_1\cap F_1$ and $U_2\cap F_2$ in $\mathcal{B}$, we can take $B=(U_1\cap U_2)\cap( F_1\cap F_2)\in \mathcal{B}$. So, we assume it generates the filter $\mathcal{F}_1$. For each $F \in \mathcal{F}$, we have $F=E\cap F\in \mathcal{B}$, and so we get $\mathcal{F}_1\subseteq\mathcal{F}_2$. Therefore, for given a zero neighborhood $U$ containing $e$, we have $U=U\cap E\in \mathcal{B}$, and so we get  $\mathcal{F}_1\tc e$.

Conversely, assume such filter $\mathcal{F}_1$ exists it means that $\mathcal{F}_1$ is a filter on the set $E$ with $\mathcal{F}\subseteq\mathcal{F}_1$ and $\mathcal{F}_1\tc e$. So, all zero neighborhoods containing $e$ is in $\mathcal{F}_1$. So, by the definition of filter, each zero neighborhood containing $e$ intersects with member of $\mathcal{F}_1$, otherwise $\emptyset \in \mathcal{F}_1$. Therefore, in particular, intersects each set in $\mathcal{F}$, and so we get $e$ is a cluster point of $\mathcal{F}$.
\end{proof}

Now, we use net to define filter on locally solid vector lattice. Let $(x_\alpha)_{\alpha\in I}$ be a net in the locally solid vector lattice $(E,\tau)$. Then we define its associated filter $\mathcal{F}$ on the set $E$ as follow; consider the tail $\hat{x}_\beta=\{x_\alpha:\alpha\in I,\alpha\geq\beta\}$ and $\mathcal{B}=\{\hat{x}_\beta:\beta\in I\}$. So, $\mathcal{B}$ is a filter base. Indeed, $(i)$ $\mathcal{B}$ is not empty; $(ii)$ every $\hat{x}_\beta\in \mathcal{B}$ is not empty since $I$ is directed set; $(iii)$ for any $\hat{x}_{\beta_1},\ \hat{x}_{\beta_2}\in \mathcal{B}$, consider the index $\beta=\max\{\beta_1,\beta_2\}$ so that $\hat{x}_\beta\subseteq \hat{x}_{\beta_1}\cap\hat{x}_{\beta_2}$ and  $\hat{x}_\beta \in \mathcal{B}$. Thus, the filter which is generated by $\mathcal{B}$ is the associated filter of $(x_\alpha)_{\alpha\in I}$.

\begin{thm}\label{net convergence and filter convergence}
Let $(E,\tau)$ be a locally solid vector lattice and $(x_\alpha)_{(\alpha\in I)}$ be a net in $E$. Assume $\mathcal{F}$ is the associated filter of $(x_\alpha)_{(\alpha\in I)}$ and $e\in E$. Then $x_\alpha\tc e$ as a net iff $\mathcal{F}\tc e$ as a filter. Moreover, $e$ is a cluster point of $(x_\alpha)_{(\alpha\in A)}$ iff $e$ is a cluster point of $\mathcal{F}$.
\end{thm}

\begin{proof}
We show only the convergence part of proof, the cluster point case is analogous. Suppose $x_\alpha\tc e$ as a net. Since every locally solid vector lattice has a base of zero neighborhoods, we can consider any zero neighborhood $U$ which contains $e$. So, by definition of $\mathcal{F}$, we get $U\in \mathcal{F}$. Thus, we have $\mathcal{F}\tc e$ as a filter. 

Conversely, assume the filter $\mathcal{F}$ converges to $e$. Let $V$ be a zero neighborhood in $E$ and contains $e$. Then $V\in \mathcal{F}$. By definition of $\mathcal{F}$, there exists a index $\alpha_0$ such that $\hat{x}_{\alpha_0}\subseteq V$. Therefore, we get $x_\alpha\in V$ for every $\alpha\geq \alpha_0$, and so  $x_\alpha\tc e$.
\end{proof}

\begin{cor}
Let $(E,\tau)$ be a locally solid vector lattice which is generated by a family of Riesz pseudonorms $\{\rho_j\}_{j\in J}$, $(x_\alpha)_{(\alpha\in A)}$ be a net in $E$ and $\mathcal{F}$ be associated filter of it, and $e\in E$. Then $\mathcal{F}\tc e$  iff $\rho_j(x_\alpha-e)\to 0$ for all $j\in J$.
\end{cor}

\begin{proof}
It follows from \cite[Thm.2.28]{ABur} and Theorem \ref{net convergence and filter convergence}.
\end{proof}

We can give convergence of filters with respect to continuous function. Let $(E,\tau)$ and $(F,\acute{\tau})$ be locally solid vector lattices, and $\mathcal{F}$ be a filter on $E$. For a function $f:E\to F$, the set
$$
\acute{f}\mathcal{F}=\{B\subset F:f^{-1}(B)\in \mathcal{F}\}
$$
is the filter on $E$ generated by $\{f(A):A\subseteq \mathcal{F}\}$.
\begin{prop}
Suppose $(E,\tau)$ and $(F,\acute{\tau})$ are locally solid vector lattices. Then a function $f:E\to F$ is continuous iff $\acute{f}\mathcal{F}\tcc f(e)$ in $F$ for each $\mathcal{F}\tc e$ in $E$.
\end{prop}

\begin{proof}
Suppose $f$ is continuous. For fixed $e\in E$, consider the following set
$$
\mathcal{N}_e=\{A\subseteq E:\exists U \ \text{zero neighborhood s.t.}, e\in U\subseteq A\}.
$$
Thus, we get $\mathcal{N}_e\tc e$. Note also that a filter $\mathcal{F}$ converges to $e$ iff $\mathcal{N}_e\subset \mathcal{F}$. Take a filter $\mathcal{F}$ such that $\mathcal{F}\tc e$. The continuity of $f$ implies that $\mathcal{N}_{f(e)}\subseteq \acute{f}\mathcal{N}_{e}$. Therefore, if $\mathcal{F}\tc e$ then $\mathcal{N}_e\subseteq \mathcal{F}$, and so $\mathcal{N}_{f(e)}\subseteq \acute{f}\mathcal{N}_{e}\subseteq \acute{f}\mathcal{F}$ so that $\acute{f}\mathcal{F}\tcc f(e)$.
	
Assume $\acute{f}\mathcal{N}_e$ is the set of all subsets of $E$ whose preimage is a neighborhood of $e$. Since $\mathcal{N}\tc e$, we conclude that the preimage of any neighborhood of $f(e)$ is a neighborhood of $e$. Hence, $f$ is continuous.
\end{proof}

%%%%%%%%%%%%%%%%%%%%%%%%%%%%%%%%%%%%%%%%%%%%%%%%%%%%%%%%%%%%%%
%%%%%%%%%%%%%%%%%%%%%%%%%%%%%%%%%%%%%%%%%%%%%%%%%%%%%%%%%%%%

\end{document}